\documentclass[12pt,reqno]{amsproc}

\usepackage{amsmath,amsthm,amssymb,wasysym,mathtools}
\usepackage[inline]{enumitem}
\usepackage{caption}
\usepackage{subcaption}
\usepackage{graphicx}

\DeclareMathOperator{\RE}{Re}

\numberwithin{equation}{section}
\theoremstyle{plain}

\newtheorem{theorem}{Theorem}[section]

\theoremstyle{definition}

\allowdisplaybreaks

\begin{document}
	
\title{Radius of Starlikeness for Bloch Functions}
	
\author[Somya Malik]{Somya Malik}
\address{Department of Mathematics \\National Institute of Technology\\Tiruchirappalli-620015,  India }
\email{arya.somya@gmail.com}
	
\author{V. Ravichandran}
\address{Department of Mathematics \\National Institute of Technology\\Tiruchirappalli-620015,  India }
\email{vravi68@gmail.com; ravic@nitt.edu}

\begin{abstract}
For normalised analytic functions  $f$ defined on the open unit disc $\mathbb{D}$ satisfying the  condition $\sup_{z\in \mathbb{D}}(1-|z^2|) |f'(z)|\leq 1$, known as   Bloch functions, we determine various starlikeness  radii. 	
\end{abstract}


\thanks{The first author is supported by  the UGC-JRF Scholarship.}

\maketitle

\section{Introduction}The class $\mathcal{A}$ consists of all analytic functions $f$ on the disc $\mathbb{D}:=\{z\in \mathbb{C}: |z|<1\}$ and normalized by the conditions $f(0)=0$ and $ f'(0)=1$. The class  $\mathcal{S}$ consists of all univalent functions $f\in \mathcal{A}$. The class $\mathcal{B}$ of Bloch functions   consists of all functions $f\in\mathcal{A}$ satisfying $ \sup_{z\in \mathbb{D}}(1-|z|^2)|f'(z)|\leq 1$ (see \cite{Bonk},   \cite{Pom}). Bonk \cite{Bonk} has shown that the radius of starlikeness  of the class $\mathcal{B}$ is the same as radius of univalence and it equals $1/\sqrt{3}\approx 0.57735$. It also follows from his distortion inequalities that the radius of close-to-convexity (with respect to $z$)  is also $1/\sqrt{3}$. We extend the radius of starlikeness result by find various other starlikeness radii.   An analytic function  $f$    is subordinate to the analytic function $g$, denoted as $f\prec g$, if there exists a function $w:\mathbb{D}\rightarrow \mathbb{D}$ with $w(0)=0$ satisfying $f(z)=g(w(z))$. If $g$ is   univalent, then $f\prec g$ if and only if $f(0)=g(0)$ and $f(\mathbb{D})\subseteq g(\mathbb{D})$. The subclass of $\mathcal{S}$ of starlike functions $\mathcal{S}^{*}$ is the collection of functions $f\in \mathcal{S}$ with the condition $\RE (zf'(z)/f(z))>0$ where $z \in \mathbb{D}$. The subclass $\mathcal{K}$ of convex functions consists of the functions in $\mathcal{S}$ with $\RE (1+zf''(z)/f'(z))>0$ for $z\in \mathbb{D}$. This characterisation gives a characterization of  these classes in terms of  the class $\mathcal{P}$ of Carath\'{e}odory functions or the functions with positive real part, comprising of analytic functions $p$ with $p(0)=1$ satisfying $\RE (p(z))>0$ or equivalently the subordination $p(z)\prec (1+z)/(1-z)$. Thus, the classes of starlike and convex functions are consist of $f\in \mathcal{A}$ with $zf'(z)/f(z) \in \mathcal{P}$ and $1+zf''(z)/f'(z) \in \mathcal{P}$ respectively. Several subclasses of starlike and convex functions were defined using subordination of $zf'(z)/f(z)$ and $1+zf''(z)/f'(z)$ to some function in $\mathcal{P}$. Ma and Minda \cite{MaMinda} gave a unified treatment of growth, distortion, rotation  and coeffcient inequalities for function in classes $\mathcal{S}^{*}(\varphi)=\{f\in \mathcal{A}:zf'(z)/f(z) \prec \varphi (z) \}$ and $\mathcal{K}(\varphi)=\{f\in \mathcal{A}:1+zf''(z)/f'(z) \prec \varphi (z)\}$, where $\varphi \in \mathcal{P}$, starlike with respect to 1, symmetric about the $x$-axis and $\varphi '(0)>0$. Numerous classes were defined for various choices of the function $\varphi$ such as $(1+Az)/(1+Bz),\ \mathit{e}^{z},\ z+\sqrt{1+z^2}$ and so on.
	
For any two subclasses $\mathcal{F}$ and $\mathcal{G}$ of $\mathcal{A}$, the $\mathcal{G}-$ radius of the class $\mathcal{F}$, denoted as $R_{\mathcal{G}} (\mathcal{F})$ is the largest number $R_{\mathcal{G}} \in (0,1)$ such that $r^{-1}f(rz)\in \mathcal{G}$ for all $f\in \mathcal{F}$ and $0<r<R_{\mathcal{G}}$. 	We determine the radius of  various subclasses of starlike functions such as starlikeness associated with the exponential function, lune, cardioid and a particular rational function for the class of Bloch Functions.

\section{Radius Problems}
In 2015, Mendiratta \emph{et al.} \cite{Exp} introduced a subclass $S^{*}_{\mathit{e}}$ of starlike functions associated with the exponential function. This class $S^{*}_{\mathit{e}}$ consists of all functions $f\in\mathcal{A}$ satisfying the subordination $ zf'(z)/f(z)\prec e^z$. This subordination  is equivalent to the inequality $|\log (zf'(z)/f(z))| <1$. Our first theorem gives the $S^{*}_{\mathit{e}}$-radius of Bloch functions.

\begin{theorem}
	The $S^{*}_{\mathit{e}}$-radius of the class $\mathcal{B}$ of Bloch functions is   \[\mathcal{R}_{S^{*}_{\mathit{e}}}(\mathcal{B})= \frac{1}{4} \sqrt{3} \left(3-3 e+\sqrt{1-10 e+9 e^2}\right) \approx 0.517387.\] The obtained radius is sharp.
\end{theorem}
\begin{proof} For functions $f \in \mathcal{B}$, Bonk   \cite{Bonk} proved the following inequality
	\begin{equation}\label{eqn1}
	\left|\dfrac{zf'(z)}{f(z)}-\frac{\sqrt{3}}{\sqrt{3}-r}\right|\leq \frac{\sqrt{3}r}{(\sqrt{3}-r)(\sqrt{3}-2r)},\ \ |z|=r<\frac{1}{\sqrt{3}}.
	\end{equation}
The function \[h(r):=\frac{\sqrt{3}}{\sqrt{3}-r}-\frac{\sqrt{3}r}{(\sqrt{3}-r)(\sqrt{3}-2r)}=\frac{3-3\sqrt{3}r}{(\sqrt{3}-r)(\sqrt{3}-2r)}\]
is a decreasing funciton of $r$ for $0\leq r<1/\sqrt{3}=\mathcal{R}_{S^{*}}(\mathcal{B})$. The number $R=\mathcal{R}_{S^{*}_{\mathit{e}}}(\mathcal{B})<1/\sqrt{3}=\mathcal{R}_{S^{*}}(\mathcal{B})$ is the smallest positive root of the polynomial
\begin{equation}\label{eqR}
 2R^2+3\sqrt{3}(\mathit{e}-1)R+3(1-\mathit{e})=0
\end{equation}  or
 $h(R)=1/\mathit{e}$. Therefore, for $0\leq r< R$, it follows that  $ 1/e=h(R)<h(r)$ and hence  \begin{align}\label{eqn3}
\frac{\sqrt{3}r}{(\sqrt{3}-r)(\sqrt{3}-2r)}<\frac{\sqrt{3}}{\sqrt{3}-r}-\frac{1}{\mathit{e}}.
\end{align}
Thus \eqref{eqn1} and \eqref{eqn3} give \begin{align}\label{eqn3a} \left|\dfrac{zf'(z)}{f(z)}-\frac{\sqrt{3}}{\sqrt{3}-r}\right|< \frac{\sqrt{3}}{\sqrt{3}-r}-\frac{1}{\mathit{e}}, \quad |z|=r< R.\end{align}

 The  function $C(r)=\sqrt{3}/(\sqrt{3}-r )$   is an increasing function of $r$, so for $r\in [0,R)$,
 it follows that $C(r) \in [1,C(R))\subseteq [1,C(0.6))\approx [1,1.53001)\subseteq (0.367879,1.54308)\approx (1/\mathit{e},(\mathit{e}+\mathit{e}^{-1})/2)$.
By  \cite[Lemma 2.2]{Exp}, for $1/\mathit{e}<c<\mathit{e}$, we have $\{w:\ \left|w-c\right|<r_c\}\subseteq \{w:\ \left|\log (w)\right|<1\}$
when $r_c$ is given by
\begin{align}\label{eqn2}
r_c &=
\begin{dcases}
c- \mathit{e}^{-1} & \text{ if }\ \mathit{e}^{-1}<c\leq \frac{\mathit{e} +\mathit{e}^{-1}}{2} ,\\
\mathit{e} -c & \text{ if }\  \frac{\mathit{e} +\mathit{e}^{-1}}{2}\leq c<\mathit{e}.
\end{dcases}
\end{align}
 By \eqref{eqn3a},  we see that $w=zf'(z)/f(z)$, $|z|<R$,  satisfies $ |w-c|<c-\mathit{e}^{-1}$ and hence it follows that $   \left|\log (w)\right|<1 $. This shows that $S^{*}_{\mathit{e}}$-radius of the class $\mathcal{B}$ is at least $R$.

We now show that $R$ is the exact $S^{*}_{\mathit{e}}$-radius of the class $\mathcal{B}$.
The function $f_0:\mathbb{D}\to\mathbb{C}$ defined by  \[f(z)=\dfrac{\sqrt{3}}{4}\left\{1-3\left(\dfrac{z-\sqrt{1/3}}{1-\sqrt{1/3}z}\right)^2\right\}\ =\ \dfrac{3z(3-2\sqrt{3}z)}{(3-\sqrt{3}z)^2}\]
is an example of function in the class $\mathcal{B}$ and it serves as an extremal function for the various  problems.
For  this function, we have \[\dfrac{zf'(z)}{f(z)}=\dfrac{3\sqrt{3}-9z}{2\sqrt{3}z^2-9z+3\sqrt{3}}.\]
Using the equation \eqref{eqR},  we get
  $2\sqrt{3}R^2-9R+3\sqrt{3}=\mathit{e} (3\sqrt{3}-9R)$, thus, for $z=R$
\begin{align*}
\left|\log \left(\dfrac{zf'(z)}{f(z)}\right)\right| &=\left|\log \left(\dfrac{3\sqrt{3}-9z}{2\sqrt{3}z^2-9z+3\sqrt{3}}\right)\right|
 = \left|\log \left(\dfrac{1}{\mathit{e}}\right)\right|=1.
\end{align*}
This proves that $R$ is the exact $S^{*}_{\mathit{e}}$-radius of the class $\mathcal{B}$.
\end{proof}

Sharma \emph{et al.} studied the class $\mathcal{S}^{*}_{c}=\mathcal{S}^{*} (\phi _c)= \ \mathcal{S}^{*} (1+(4/3)z+(2/3)z^2)$ and gave  \cite[Lemma 2.5]{Cardioid} \\
For $1/3<c<3,$
\begin{align}\label{eqn4}
r_c &=
\begin{dcases}
\frac{3c-1}{3} & \text{ if }\ \frac{1}{3}<c\leq \frac{5}{3}\\
3-c & \text{ if } \ \frac{5}{3}\leq c<3
\end{dcases}
\end{align}

then $\{w: |w-c|<r_c\} \subseteq \Omega _c$. Here $\Omega_c$ is the region bounded by the cadioid $\{x+\iota y: (9x^2+9y^2-18x+5)^2 -16(9x^2+9y^2-6x+1)=0\}.$

\begin{theorem} \
	The $S^{*}_{c}$-radius $\mathcal{R}_{S^{*}_{c}}\approx 0.524423.$ This radius is sharp.
\end{theorem}

\begin{proof}
		$R=\mathcal{R}_{S^{*}_{c}}$ is the smallest positive root of the equation \[R^2+3\sqrt{3}R-3=0.\]
	The function \[h(r):=\frac{\sqrt{3}}{\sqrt{3}-r}-\frac{\sqrt{3}r}{(\sqrt{3}-r)(\sqrt{3}-2r)}=\frac{3-3\sqrt{3}r}{(\sqrt{3}-r)(\sqrt{3}-2r)}\]
	is a decreasing funciton of $r$ for $0\leq r<1/\sqrt{3}=\mathcal{R}_{S^{*}}$. \cite[Corollary, P.455]{Bonk} Note that the class $\mathcal{S}^{*}_{c}$  is a subclass of the parabolic starlike class $S^{*}$, Also since, $R=\mathcal{R}_{S^{*}_{c}}$ is the smallest positive root of the equation $h(r)=1/3$. For $0\leq r< R$, we have
	\begin{align}\label{eqn5}
	\frac{\sqrt{3}r}{(\sqrt{3}-r)(\sqrt{3}-2r)}<\frac{\sqrt{3}}{\sqrt{3}-r}-\frac{1}{3}
	\end{align}
	Thus \eqref{eqn1} and \eqref{eqn5} give \[\left|\dfrac{zf'(z)}{f(z)}-\dfrac{1}{1-ar}\right|< \frac{1}{1-ar}-\frac{1}{3}; \ |z|\leq r,\ a=\frac{1}{\sqrt{3}}.\]
	The center $C(r)$ of \eqref{eqn1} is an increasing function of $r$, so for $r\in [0,R),\ C(r) \in [1,C(R))\subseteq [1,c(0.6))\approx [1,1.53001)\subseteq (1/3,5/3)$. Now, by \eqref{eqn4} we get that the disc $\{w: |w-c|<c-1/3\} \subseteq \Omega _c. $
	\\
	\\
	For proving sharpness, consider the function \[f(z)=\dfrac{\sqrt{3}}{4}\left\{1-3\left(\dfrac{z-\sqrt{1/3}}{1-\sqrt{1/3}z}\right)^2\right\}\]
	for this function, $\dfrac{zf'(z)}{f(z)}=\dfrac{3\sqrt{3}-9z}{2\sqrt{3}z^2-9z+3\sqrt{3}},$\
	and using the equation for $R$, we get
	\\ $2\sqrt{3}r^2-9r+3\sqrt{3}=\ 3(3\sqrt{3}-9r)$, thus for $z=R$
	\begin{align*}
	\dfrac{zf'(z)}{f(z)}
	&= \dfrac{1}{3}\\
	&= \phi _c (-1).
	\end{align*}	
\end{proof}

The class $\mathcal{S}^{*}_{\leftmoon}=\mathcal{S}^{*} (z+\sqrt{1+z^2})$ was introduced in 2015 by Rain and Sok\'{o}l \cite{Sokol} in 2015 and proved that $f\in \mathcal{S}^{*}_{\leftmoon} \iff zf'(z)/f(z)$ lies in the lune region $\{w: |w^2-1|<2|w|\}$. Gandhi and Ravichandran \cite[Lemma 2.1]{Lune} proved that
\\
for $\sqrt{2}-1<c\leq \sqrt{2}+1,$
\begin{align}\label{eqn6}
\{w: |w-c|<1-|\sqrt{2}-c|\}\subseteq \{w: |w^2-1|<2|w|\}
\end{align}

\begin{theorem}
	The $\mathcal{S}^{*}_{\leftmoon}$ radius, $\mathcal{R}_{\mathcal{S}^{*}_{\leftmoon}} \approx 0.507306.$ The radius is sharp.
\end{theorem}

\begin{proof}
	 For $R=\mathcal{R}_{\mathcal{S}^{*}_{\leftmoon}}$, the center of \eqref{eqn1},\  $C(R)=\sqrt{2};$ since $C(r)$ is an increasing function of $r$, thus for $0\leq r< R,\ 1\leq C(r)<\sqrt{2}$, or\[ \text{for}\  0\leq r<R,\ \sqrt{2}-C(r)\geq 0.\]
	  So, $R=\mathcal{R}_{\mathcal{S}^{*}_{\leftmoon}}$ is the smallest positive root of the equation \[(2-2\sqrt{2})R^2+\sqrt{3}(3\sqrt{2}-6)R+3(2-\sqrt{2})=0.\]
	
	  The function \[h(r):=\frac{\sqrt{3}}{\sqrt{3}-r}-\frac{\sqrt{3}r}{(\sqrt{3}-r)(\sqrt{3}-2r)}=\frac{3-3\sqrt{3}r}{(\sqrt{3}-r)(\sqrt{3}-2r)}\]
	  is a decreasing funciton of $r$ for $0\leq r<1/\sqrt{3}=\mathcal{R}_{S^{*}}$. \cite[Corollary, P.455]{Bonk} Note that the class $\mathcal{S}^{*}_{\leftmoon}$  is a subclass of the parabolic starlike class $S^{*}$. Also since, $R=\mathcal{R}_{\mathcal{S}^{*}_{\leftmoon}}$ is the smallest positive root of the equation $h(r)=\sqrt{2}-1$. For $0\leq r< R$, we have
	  \begin{align}\label{eqn7}
	  \frac{\sqrt{3}r}{(\sqrt{3}-r)(\sqrt{3}-2r)}<1-\sqrt{2} +\frac{\sqrt{3}}{\sqrt{3}-r}=1-\left|\sqrt{2}-\frac{\sqrt{3}}{\sqrt{3}-r}\right|.
	  \end{align}
	  Thus \eqref{eqn1} and \eqref{eqn7} give \[\left|\dfrac{zf'(z)}{f(z)}-\dfrac{1}{1-ar}\right|<1-\left|\sqrt{2}-\frac{1}{1-ar}\right|; \ |z|\leq r,\ a=\frac{1}{\sqrt{3}}.\]
	  The center $C(r)$ of \eqref{eqn1} is an increasing function of $r$, so for $r\in [0,R),\ C(r) \in [1,C(R))\subseteq [1,c(0.6))\approx [1,1.53001)\subseteq (\sqrt{2}-1,\sqrt{2}+1)$. Now, by \eqref{eqn6} we get that the $R$ is the required radius.
	\\
	\\
	\\
	Consider the  \[f(z)=\dfrac{\sqrt{3}}{4}\left\{1-3\left(\dfrac{z-\sqrt{1/3}}{1-\sqrt{1/3}z}\right)^2\right\}\]
	for this function, $\dfrac{zf'(z)}{f(z)}=\dfrac{3\sqrt{3}-9z}{2\sqrt{3}z^2-9z+3\sqrt{3}},$\\
	and we can easily see that for $z=\frac{1}{2}[2\sqrt{3}-\sqrt{6}],$
	\[\left|\left(\dfrac{zf'(z)}{f(z)}\right)^2-1\right|=\ 2\left(\dfrac{zf'(z)}{f(z)}\right)\ =\ 2(\sqrt{2}-1).\] Thus, the result is sharp.
\end{proof}

The next class that we consider is the class of starlike functions associated with a rational function. Kumar and Ravichandran \cite{Rational} introduced the class of starlike functions associated with the rational function $\psi (z)=1+ ((z^2+kz)/(k^2-kz))$ where $k=\sqrt{2}+1$, denoted by $\mathcal{S}^{*}_{R}=\mathcal{S}^{*}(\psi(z))$ They proved\cite[Lemma 2.2]{Rational} that \\
for $2(\sqrt{2}-1)<c<2,$
\begin{align}\label{eqn8}
r_c &=
\begin{dcases}
c-2(\sqrt{2}-1)\ & \text{ if }\ 2(\sqrt{2}-1)<c\leq \sqrt{2}\\
2-c\ & \text{ if }\ \sqrt{2}\leq c<2
\end{dcases}
\end{align}

then $\{w: |w-c|<r_c\} \subseteq \psi (\mathbb{D})$

\begin{theorem}
	The $\mathcal{S}^{*}_{R}$ radius is the smallest positive root of the polynomial $4(1-\sqrt{2})r^2+3\sqrt{3}(2\sqrt{2}-3)r+3(3-2\sqrt{2})$ that is  $\mathcal{R}_{\mathcal{S}^{*}_{R}} \approx 0.349865.$ The result is sharp.
\end{theorem}
\begin{proof}
		$R=\mathcal{S}^{*}_{R}$ is the smallest positive root of the equation \[4(1-\sqrt{2})R^2+3\sqrt{3}(2\sqrt{2}-3)R+3(3-2\sqrt{2})=0.\]
	The function \[h(r):=\frac{\sqrt{3}}{\sqrt{3}-r}-\frac{\sqrt{3}r}{(\sqrt{3}-r)(\sqrt{3}-2r)}=\frac{3-3\sqrt{3}r}{(\sqrt{3}-r)(\sqrt{3}-2r)}\]
	is a decreasing funciton of $r$ for $0\leq r<1/\sqrt{3}=\mathcal{R}_{S^{*}}$. \cite[Corollary, P.455]{Bonk} Note that the class $\mathcal{S}^{*}_{R}$  is a subclass of the parabolic starlike class $S^{*}$. Also since, $R=\mathcal{S}^{*}_{R}$ is the smallest positive root of the equation $h(r)=2(\sqrt{2}-1)$. For $0\leq r< R$, we have
	\begin{align}\label{eqn9}
	\frac{\sqrt{3}r}{(\sqrt{3}-r)(\sqrt{3}-2r)}<\frac{\sqrt{3}}{\sqrt{3}-r}-2(\sqrt{2}-1)
	\end{align}
	Thus \eqref{eqn1} and \eqref{eqn9} give \[\left|\dfrac{zf'(z)}{f(z)}-\dfrac{1}{1-ar}\right|< \frac{1}{1-ar}-2(\sqrt{2}-1); \ |z|\leq r,\ a=\frac{1}{\sqrt{3}}.\]
	The center $C(r)$ of \eqref{eqn1} is an increasing function of $r$, so for $r\in [0,R),\ C(r) \in [1,C(R))\subseteq [1,c(0.4))\approx [1,1.30029)\subseteq (2(\sqrt{2}-1),\sqrt{2})$. Now, by \eqref{eqn8} we get that the disc $\{w: |w-c|<x-2(\sqrt{2}-1)\} \subseteq \psi (\mathbb{D}). $
	\\
	\\
	To show that the result is sharp, consider the function \[f(z)=\dfrac{\sqrt{3}}{4}\left\{1-3\left(\dfrac{z-\sqrt{1/3}}{1-\sqrt{1/3}z}\right)^2\right\}\]
	for this function, $\dfrac{zf'(z)}{f(z)}=\dfrac{3\sqrt{3}-9z}{2\sqrt{3}z^2-9z+3\sqrt{3}},$\
	and using the equation for $R$ we get
	\\ $3\sqrt{3}-9r=(2\sqrt{2}-2)(2\sqrt{3}r^2-9r+3\sqrt{3})$ thus for $z=R$
	\begin{align*}
	\dfrac{zf'(z)}{f(z)}
	&= 2\sqrt{2}-2\\
	&= \psi(-1).
	\end{align*}
\end{proof}

\begin{theorem}
	For the class $\mathcal{B}$ the following results hold:
	\begin{enumerate}
		\item The Lemniscate starlike radius,\  $R_{\mathcal{S}^{*}_{L}}=\frac{2\sqrt{3}-\sqrt{6}}{4}\approx 0.253653.$
		\item The starlike radius associated with the sine function,\ $R_{\mathcal{S}^{*}_{sin}}=\frac{\sqrt{3}\sin 1}{2+2\sin 1}\approx 0.395735.$
		\item The nephroid radius, \ $R_{\mathcal{S}^{*}_{Ne}}=\frac{\sqrt{3}}{5}\approx 0.34641.$
		\item The sigmoid radius,\ $R_{\mathcal{S}^{*}_{SG}}=\frac{\sqrt{3}(\mathit{e}-1)}{4\mathit{e}}\approx 0.273716.$
	\end{enumerate}
 
\end{theorem}


\begin{thebibliography}{9}
	\bibitem{Bonk}  M. Bonk, Distortion estimates for Bloch functions, Bull. London Math. Soc. {\bf 23} (1991), no.~5, 454--456.
	
	\bibitem{Exp} R. Mendiratta, S. Nagpal\ and\ V. Ravichandran, On a subclass of strongly starlike functions associated with exponential function, Bull. Malays. Math. Sci. Soc. {\bf 38} (2015), no.~1, 365--386.
	
	\bibitem{Cardioid}K. Sharma, N. K. Jain\ and\ V. Ravichandran, Starlike functions associated with a cardioid, Afr. Mat. {\bf 27} (2016), no.~5-6, 923--939.
	
	\bibitem{Sokol}R. K. Raina\ and\ J. Sok\'{o}\l, Some properties related to a certain class of starlike functions, C. R. Math. Acad. Sci. Paris {\bf 353} (2015), no.~11, 973--978
	
	\bibitem{Lune} S. Gandhi\ and\ V. Ravichandran, Starlike functions associated with a lune, Asian-Eur. J. Math. {\bf 10} (2017), no.~4, 1750064, 12 pp.
	
	\bibitem{Rational}S. Kumar\ and\ V. Ravichandran, A subclass of starlike functions associated with a rational function, Southeast Asian Bull. Math. {\bf 40} (2016), no.~2, 199--212.
	
	\bibitem{MaMinda}W. C. Ma\ and\ D. Minda, A unified treatment of some special classes of univalent functions, in {\it Proceedings of the Conference on Complex Analysis (Tianjin, 1992)}, 157--169, Conf. Proc. Lecture Notes Anal., I, Int. Press, Cambridge, MA.
	
	\bibitem{Pom}Ch. Pommerenke, On Bloch functions, J. London Math. Soc. (2) {\bf 2} (1970), 689--695.
\end{thebibliography}
\end{document}